\documentclass{amsart}
\usepackage{amsmath}
  \usepackage{paralist}
  \usepackage{amssymb} 
  \usepackage{graphics} 
  \usepackage{epsfig} 
 \usepackage[colorlinks=true]{hyperref}

\hypersetup{urlcolor=blue, citecolor=red}

\numberwithin{equation}{section}

\newcounter{mnote}
\newtheorem{theorem}{Theorem}[section]
\newtheorem{corollary}[theorem]{Corollary}
\newtheorem{lemma}[theorem]{Lemma}
\newtheorem{proposition}[theorem]{Proposition}



\def\csubA{c_{\text{A}}}

\def\csubL{c_{\text{L}}}
\def\csubS{c_{\text{S}}}

\def\cA{\mathcal{A}}

\def\calA{\mathcal{A}}

\def\calL{\mathcal{L}}
\def\calS{\mathcal{S}}
\def\calV{\mathcal{V}}

\def\cP{\mathcal{P}}

\def\k0{\kappa_0}

\def\lgl{\langle}
\def\rgl{\rangle}

\def\bR{{\mathbb R}}

\def\lgl{\langle}
\def\rgl{\rangle}

\def\mF{\Phi}

\def\tc{\tilde{c}}
\newcommand{\B}[3]{\left(B\left(#1,#2\right),#3\right)}

\newcommand{\de}{\delta}
\newcommand{\De}{\Delta}

\begin{document}

\title[Unified approach to determining
forms]
{A unified approach to determining
forms for the 2D Navier-Stokes equations -- the general interpolants case}
\author[C. Foias]{Ciprian Foias$^{1}$}
\address{$^1$Department of Mathematics\\
Texas A\&M University\\ College Station, TX 77843, USA.}
\author[M.S. Jolly]{Michael S. Jolly$^{2}$}
\address{$^2$Department of Mathematics\\
Indiana University\\ Bloomington, IN 47405, USA.}
\author[R. Kravchenko]{Rostyslav Kravchenko$^{3,\dagger}$}
\address{$^3$ Department of Mathematics\\
Chicago University \\Chicago, IL 60637, USA. }
\author[E. S. Titi]{Edriss S. Titi$^{4}$}
\address{$^4$ Department of Computer Science and Applied Mathematics\\
Weizmann Institute of Science\\
Rehovot, 76100, Israel.\\ ALSO: Department of Mathematics and Department
of Mechanical and Aerospace Engineering \\
University of California \\
Irvine, California 92697, USA.
}
\address{$\dagger$ corresponding author}
\email[E. S. Titi] {etiti@math.uci.edu}

\date{November 7, 2013. {\bf To appear in:} {\it Russian Mathematical Surveys.}}
\begin{abstract}

In this paper we show that the long time dynamics (the global attractor) of the 2D Navier-Stokes equation is embedded in the long time dynamics  of an ordinary differential equation, named {\it determining form}, in a space of trajectories which is isomorphic to $C^1_b(\bR; \bR^N)$, for $N$ large enough depending on the physical parameters of the Navier-Stokes equations. We present a unified approach based on  interpolant operators that are induced by any of the determining parameters for the Navier-Stokes equations, namely, determining nodal values, Fourier modes, finite volume elements,  finite elements, etc... There are two immediate and interesting consequences of this unified approach. The first is that the constructed determining form has a Lyapunov function, thus its solutions converge, as time goes to infinity,  to the set of steady states of the determining form. The second is that these steady states of the determining form are identified, one-to-one, with the trajectories on the global attractor of the Navier-Stokes equations. It is worth adding that this unified  approach is general enough that it applies, in an almost straightforward manner, to a whole class of dissipative dynamical systems.

\end{abstract}
\maketitle

\begin{center}
{\it Dedicated to the memory of Professor Mark Vishik}
\end{center}

\section{Introduction}

The 2D Navier-Stokes equation (NSE), \eqref{NSES} and \eqref{NSE}, in addition to being a fundamental component of
many fluid models, is intriguing for several theoretical reasons.  In featuring both a direct cascade of
enstrophy and an inverse cascade of energy, it displays more complicated turbulence phenomena
than does the 3D NSE \cite{Batchelor,Kraichnan,Leith}.  Also unlike in 3D, the global existence theory
for the 2D NSE is complete (see, e.g., \cite{CF88,T97}).  In fact, the long time dynamics of
the 2D NSE is entirely contained in the {\it global attractor} $\calA$,  \eqref{attractordef}, a compact finite-dimensional set within the  infinite-dimensional phase space $H$, of solenoidal finite energy vector fields, (see, e.g., {\cite{CF88,Foias2001,Hale,T97}). Sharp estimates concerning the dimension of the global attractor in terms of the relevant physical parameters were first established in  \cite{CFT89} (see also \cite{CF88,Liu,T97} and references therein).   If there were an {\it inertial manifold}, $\mathcal{M}$, i.e., a Lipschitz, finite-dimensional, forward invariant manifold, which attracts each bounded set at an exponential rate,    then  $\calA \subset \mathcal{M}$; and the dynamics on $\mathcal{A}$ would be captured  by an ordinary differential equation (ODE), called an {\it inertial form}, in a
finite-dimensional phase space \cite{CF88,FST,FSTi,T97}.  This is achieved   through the reduction of the original evolution equation to the inertial manifold $\mathcal{M}$.   Yet the existence of an inertial manifold for the 2D NSE has been an problem open since the 1980s!!

This is rather surprising as there are even stronger indicators of finite dimensional behavior
for the 2D NSE.  Solutions in $\calA$ are determined by the asymptotic behavior of a sufficient, finite number determining parameters.   If, in the limit, as $t \to \infty$, a certain large enough number of low Fourier modes, or  nodal values, or finite volume elements for two solutions in $\calA$ converge to each other, then those solutions coincide (see, e.g., \cite{Cockburn-Jones-Titi} for a unified theory of determining parameters and projections). This is equivalent, at least in the case of Fourier modes, to the following:
\begin{equation}\label{equiv def}
\begin{aligned}
&\text{If two complete trajectories in the global attractor coincide}\\
& \text{upon projection $P_m$ onto a sufficient large number, $m$,} \\
& \text{of the low modes, then they are the same trajectory.}
\end{aligned}
\end{equation}

This notion of {\it determining modes}, which was introduced in \cite{FP} (see also, \cite{Jones-Titi1993} for sharp estimates regarding the number of determining modes), was used in \cite{FJKrT1} to construct a system of ODE in the Banach space $X=C_b(\bR, P_m H)$, governing the evolution of trajectories in the space $X$. We call that system of ODEs
a {\it determining form}.   Trajectories in the global attractor, $\mathcal{A}$, of the 2D NSE, \eqref{NSE}, are identified with traveling wave solutions of that  determining form.
There are, conceptually, two time variables in play for the determining form:
the evolving time for the ODE
and the original time variable of the NSE that now parameterizes complete trajectories
in $X$.  Though that determining form has an infinite dimensional phase space, the vector field that governs the evolution is globally Lipschitz; so the determining form is an ODE
in the true sense.  The key to constructing that determining form in \cite{FJKrT1} is to extend the mapping provided by \eqref{equiv def} $W:P_m\calS \to (I-P_m)\calS$ on the {\it set} $\calS$ of complete trajectories in $\calA$, to the space $X$.  The extended map is shown to be Lipschitz,
and its image plays the role of recovering  the higher modes, while the evolving trajectory in $X$ represents the lower modes.

The determining form in this paper has an entirely different character.
It is a  system which possesses a Lyapunov function and whose steady states are precisely the trajectories
in the global attractor of the 2D NSE.  It is more general in that it can be used with
a variety of determining parameters, including nodal values, as well as Fourier modes. Furthermore, it provides a general framework and strategy that can be implemented for other dissipative systems. Like the determining form in \cite{FJKrT1}, the key to its construction is the extension of a map $W$, defined at first only for projections of trajectories in $\calA$.
This is done using the feedback control term added to the NSE suggested in \cite{Azouani-Olson-Titi,ATi},
which involves an interpolating operator $J_h$ approximating the identity map at the level $h$  (for instance, $J_h$, can be based on nodal values, where $h$ represents the grid size). This construction and the statements of our main results are presented in section \ref{Statement-Main-Results}. In section \ref{Funct-Setting} we provide some preliminary background material and useful inequalities concerning the Navier-Stokes equations. In sections \ref{Prop-equality} and \ref{MainProof} are dedicated for the details of the proofs of our main results.


\section{Functional Setting and the Navier-Stokes Equations\label{Funct-Setting}}
We consider the  two-dimensional incompressible Navier-Stokes equations (NSE)
\begin{equation}\label{NSES}
\begin{aligned}
&\frac{\partial u}{\partial t} -
\nu \Delta u  + (u\cdot\nabla)u + \nabla p = \mF \\
&\text {div} u = 0 \\
& \int_{\Omega} u\,  dx =0 \;,\qquad \int_{\Omega} \mF\,  dx =0 \\
& u(0,x) = u_0(x),
\end{aligned}
\end{equation}
subject to periodic boundary conditions with basic domain $\Omega =[0,L]^2$. The  velocity field, $u$, and the pressure, $p$, are the unknown functions, while $\mF$ is a given forcing term, and $\nu>0$ is a given constant viscosity.

Let us denote by
$$
\calV=\{\phi:\, \phi\,\, \text{is}\,\, \mathbb{R}^2-{\hbox{valued trigonometric polynomials}}, \,\, \nabla \cdot \phi = 0, \, \text{ and} \, \int_{\Omega} \phi(x)\, dx =0  \},
$$
For any  subset $Z \subset  L^1_{\rm{per}}(\Omega)$, we will denote  by $\dot{Z}=\{\phi\in Z:\, \int_\Omega \phi(x)\,dx =0\}$. We denote by $H$ and $V$ the closures of $\calV$ in  $(L^2_{\rm{per}}(\Omega))^2$  and $(H^1_{\rm{per}}(\Omega))^2$, respectively. The inner product and norm in the Hilbert spaces  $(L^2_{\rm{per}}(\Omega))^2$ and $H$ will be denoted by $(\cdot,\cdot)$ and $|\cdot|$, respectively; and the corresponding inner product and norm in the Hilbert spaces $(\dot{H}^1_{\rm{per}}(\Omega))^2$ and $V$ will be denoted by $((\cdot,\cdot))$ and
$\|\cdot\|$, respectively. Specifically, for every $u,v \in (\dot{H}^1_{\rm{per}}(\Omega))^2$ we set:
$$
((u,v)) = \sum_{i,j=1}^2 \int_\Omega \frac{\partial u_i(x)}{\partial x_j} \frac{\partial v_i(x)}{\partial x_j} \, dx.
$$
We will also denote by $V'$ the dual space of the space $V$.

Using the above functional notation the Navier-Stokes equations can
 be written as an  evolution equation in the Hilbert space $H$   (cf.  \cite{CF88,T97})
%
\begin{equation}\label{NSE}
\begin{aligned}
&\frac{d}{dt}u(t) + \nu Au(t) + B(u(t),u(t)) = f, \,\, \text{for}\, t >0,\\
& u(0)=u_0.
\end{aligned}
\end{equation}
The Stokes operator $A$, the bilinear operator $B$, and force $f$ are defined as
\begin{equation}
\label{nabla1}
A=-\cP\Delta\;, \quad  B(u,v)=\cP\left( (u \cdot \nabla) v \right)\;,\quad f=\cP \mF\;,
\end{equation}
where $\cP$ is the Helmholtz-Leray orthogonal projector from
$(\dot{L}^2_{\rm{per}}(\Omega))^2$ onto $H$, and where $u,v$ are sufficiently smooth such that $B(u,v)$ makes sense. In this paper we will consider  $f \in V$.

We observe that $D(A)=(\dot{H}^2_{\rm{per}}(\Omega))^2\cap V$. The operator $A$ is self-adjoint, with compact inverse. Therefore, the space $H$ possesses an orthonormal basis $\{w_j\}_{j=1}^\infty$ of eigenfunctions of $A$, namely,  $A w_j=\lambda_j w_j$, with $0<\lambda _1=\left(2\pi/ L\right)^2
\leq \lambda _2\leq\lambda_3 \le \cdots$ \; (cf. \cite{CF88,T97}).
The powers, $A^\alpha$, are defined
$$
A^\alpha v = \sum_{j=1}^\infty \lambda_j^\alpha (v,w_j) w_j.
$$
Observe that all powers of $A$ commute with $\cP$. We will also observe that  $V:=D(A^{1/2})$ and that
\begin{equation}\label{Vnorm}
\|u\|=|A^{1/2}u|= \left(\sum_{j=1}^{\infty} \lambda_j(u,w_j)^2 \right)^{1/2}.
\end{equation}

It is well known that the NSE \eqref{NSE} has a  global attractor
\begin{equation}\label{attractordef}
\mathcal{A}=\{u_0\in H: \, \exists \, \text{ a solution}\, u(t,u_0)\, \text{of \eqref{NSE}} \, \, \forall t\in\mathbb{R},\,\, \sup_t\|u(t)\|<\infty\};
\end{equation}
that is, $\mathcal{A}$ is the maximal bounded invariant subset in $V$ under the NSE dynamics, or equivalently it is the minimal compact set in
$V$ which attracts uniformly all bounded sets of $V$ under the dynamics of \eqref{NSE}. In particular, it is also known that
\begin{equation}\label{Grashof}
\mathcal{A} \subset \{u\in V: \|u\|\leq  G\nu\kappa_0\}, \quad \text{where} \quad G=\frac{|f|}{\nu^2\k0^2}.
\end{equation}
$G$ is the  Grashof number, a dimensionless physical parameter,  and $\k0=\lambda_1^{1/2}=2\pi/L$. For the above properties see, e.g., \cite{CF88,Foias2001,Hale,T97}.

Next, we introduce number of identities satisfied by the bilinear term.
This includes  the orthogonality relations
\begin{equation}\label{flip}
\lgl B(u,v),w \rgl = -\lgl B(u,w),v\rgl\;, \quad u,v,w \in V\;,
\end{equation}
where $\lgl \cdot, \cdot \rgl$ denotes the dual action between $V'$ and $V$;
and
\begin{equation}\label{ortho}
(B(u,u),Au) = 0\;,\quad u \in D(A)\;,
\end{equation}
(see, e.g., \cite{CF88,Foias2001,T97}).
Relation  \eqref{flip} implies
(cf. \cite{CF88,T97})
\begin{equation}\label{moveu}
(B(v,v),Au) + (B(v,u),Av) + (B(u,v),Av) = 0\;,\quad u,v \in D(A)\;.
\end{equation}

We note that, hereafter, $c,c_A,c_B,c_L,c_S,c_T,c_1,c_2,\tilde{c}_1, \tilde{c}_2,\dots$ will denote universal dimensionless positive constants. Our estimates for the nonlinear term   will involve Agmon's inequality
\begin{equation}\label{Agmon}
\|u\|_{\infty} \leq \csubA|u|^{1/2} |Au|^{1/2}\;, \quad u \in D(A)\;,
\end{equation}
the Sobolev and Ladyzhenskaya  inequalities
\begin{align}
&\left\| u\right\|_{L^4(\Omega)} \le \csubS\left\| u\right\|_{H^{1/2}(\Omega)} \quad {\hbox{for every}} \, u \in (\dot{H}_{\rm{per}}^{1/2}(\Omega))^2, \label{Sobolev}\\
&\left\| u\right\|_{H^{1/2}(\Omega)} \le \tilde{\csubL} \left|u\right|^{1/2} \left\|u\right\|^{1/2} \quad {\hbox{for every}} \, u \in (\dot{H}_{\rm{per}}^1(\Omega))^2 \label{LadyzhenskayaTilde}\;,
\end{align}
which yields
\begin{align}
\left\| u\right\|_{L^4(\Omega)} \le \csubL \left|u\right|^{1/2} \left\|u\right\|^{1/2} \quad {\hbox{for every}} \, u \in (\dot{H}_{\rm{per}}^1(\Omega))^2 \label{Ladyzhenskaya}\;.
\end{align}
We will use the following versions of the Poincar\'e inequality
\begin{equation}\label{fuit}
\kappa_0|v|\leq\|v\|, \quad \forall \  u\in V, \quad  {\hbox{and}} \quad \kappa_0\|v\|\leq|Av|, \quad
\forall \ u\in D(A)\;,
\end{equation}
as well as Young's inequality
\begin{equation}\label{Young}
ab \le \frac{a^p}{p}+\frac{b^q}{q} \quad
\text{for} \quad a,b,p,q > 0 \quad \text{and} \quad  \frac{1}{p}+\frac{1}{q}=1\;.
\end{equation}

By virtue of \eqref{Ladyzhenskaya} one has
\begin{equation}\label{A46a}
|(B(u,v),w)|\leq \csubS|u|^{1/2}\|u\|^{1/2}\|v\|^{1/2}|Av|^{1/2}|w|,\
\forall \ u\in V, v\in D(A), w \in H,
\end{equation}
and thanks to \eqref{Agmon}
\begin{equation}\label{A46b}
|(B(u,v),w)|\leq \csubA|u|^{1/2}|Au|^{1/2}\|v\||w|,
\quad \forall \,u\in D(A), v\in V, w \in H.
\end{equation}
In addition, we have
\begin{equation}\label{Titi_1}
|(B(w,u),v)|\leq c_T\|w\|\|u\|\left(\log\frac{e|Au|}{\kappa_0\|u\|}\right)^{1/2}|v|,\ \forall \ u\in D(A), v\in H, w \in V,
\end{equation}
(see \cite{Titi}).
Using the Br\'ezis-Gallouet inequality \cite{Brezis-Gallouet} (see also a different proof \cite{Titi}) one also has
\begin{equation}\label{Brezis}
|(B(w,u),v)|\leq c_B\|w\|\|u\|\left(\log\frac{e|Aw|}{\kappa_0\|w\|}\right)^{1/2}|v|, \ \forall \ u\in V, v\in H, w \in D(A).
\end{equation}

We will use the following modified Gronwall inequality from \cite{Jones-Titi1992} (see also \cite{Foias2001}).
\begin{lemma}  \label{modGronwall}
Let $\alpha$, $\beta$ be locally integrable real-valued functions on $(0,\infty)$,
satisfying for some $T\in(0, \infty)$
\begin{align*}
&\liminf_{t\to \infty} \frac{1}{T} \int_t^{t+T} \alpha(\tau) \ d\tau = \gamma > 0\;, \quad
\limsup_{t\to \infty} \frac{1}{T} \int_t^{t+T} \alpha^-(\tau) \ d\tau < \infty\;, \\
& \text{and} \qquad  \lim_{t\to \infty} \frac{1}{T} \int_t^{t+T} \beta^+(\tau) \ d\tau =0\;,
\end{align*}
where $\alpha^-=\max\{-\alpha, 0\}$ and $\beta^+=\max\{\beta,0\}$.  Suppose
that $\xi$ is an absolutely continuous, non-negative function on $(0,\infty)$ such that
$$
\frac{d}{dt} \xi + \alpha \xi \le \beta, \quad \text{a.e. on} \ (0, \infty)\;.
$$
Then $\xi(t) \to 0$ as $t \to \infty$.
\end{lemma}
Lemma \ref{modGronwall} will be combined later with the following estimates for averaged solutions (see \cite{Jones-Titi1992,Jones-Titi1993}).
\begin{proposition}Let $u$ be a solution of the NSE \eqref{NSE}, and let $T=(\nu\k0^2)^{-1}$.   Then
\begin{align} \label{aveAu}
\limsup_{t\to \infty} \frac{1}{T} \int_t^{t+T} |Au(\tau)|^2 \ d\tau \le 2\nu^2\k0^4G^2 \;.
\end{align}
If $u\in\cA$, then
\begin{align}\label{limsupaveAu}
\limsup_{t_0\to -\infty}\frac{1}{t-t_0}\int_{t_0}^t |Au(\tau)|^2\ d\tau \le \nu^2\k0^4 G^2, \quad\text{ for all }\  t \in \bR
\end{align}
Moreover, it follows from the Cauchy-Schwarz inequality that
\begin{align}\label{limsupintAu}
\limsup_{t_0\to -\infty} \frac{1}{t-t_0}\int_{t_0}^t |Au(\tau)| \ d\tau \le \nu \k0^2 G, \quad \text{for all} \ t \in \bR \;.
\end{align}
\end{proposition}

\begin{proposition}\label{limsupAu}
Let $u(t)$ be a solution of the NSE \eqref{NSE}, then
\begin{align*}
\limsup_{t\to \infty} \|u(t)\| \le \nu\k0 G, \quad \text{and}\quad \limsup_{t\to \infty} |Au(t)| \le c\nu\k0^2G^3.
\end{align*}
In particular, we have
\begin{equation}\label{new anal}
\|u(t)\| \le \nu\k0 G \quad \text{and} \quad  |Au|\leq c_0\nu\k0^2G^3\;, \forall u \in \cA \;.
\end{equation}
Moreover, the solutions in the global attractor $\cA$ are analytic with respect to the time variable in a strip about the real axis with width $\delta_{\rm{Time}} \ge \frac{c}{\nu \k0^2 G^4}$.
In addition, by the Cauchy formula one obtains from the above estimates that
$$
\sup_{t \in \bR} \| \frac{du}{dt}(t)\| \le c \nu^2 \k0^3 G^5 \quad \text{and} \quad \sup_{t \in \bR} |A \frac{du}{dt}(t)| \le c \nu^2 \k0^4 G^7.
$$
\end{proposition}
 The ideas of the proof of the above proposition can be found in \cite{CF88}. However, the new sharp estimates stated  in Proposition \ref{limsupAu} are obtained in \cite{FJLRYZ}.

We now derive two bounds for $A^{1/2}B(u,v)$.
\begin{proposition} For all  $u \in D(A^{3/4})$, and  $v \in D(A)$ we have
\begin{equation}\label{1_2}
\left|A^{1/2}B(u,v)\right| \leq \ c\left(|A^{3/4}u||A^{3/4}v|+|u|^{1/2}|Au|^{1/2}|Av|\right),
\end{equation}
and for all $u \in V$ and $v \in D(A^{3/2})$
\begin{equation}\label{1_3}
\left|A^{1/2}B(u,v)\right| \leq c\left(|A^{1/2}u||A^{1/2}v|^{1/2}|A^{3/2}v|^{1/2}+|A^{1/4}u||A^{5/4}v|\right).
\end{equation}
\end{proposition}
\begin{proof}
First, we observe that
\begin{equation}\label{1_1}
\left|A^{1/2}\cP \phi \right|\leq |\nabla \phi|,\quad {\hbox {for every}}\,  \phi \in (\dot{H}^1_{\rm{per}}(\Omega))^2.
\end{equation}

By virtue of \eqref{Sobolev} we have
\begin{equation}\label{grad-nonl}
((\partial_i u\cdot \nabla) v,w) \le \|\partial_i u\|_{L^4}\|\nabla v\|_{L^4}|w| \le
\|\partial_i u\|_{H^{1/2}}\|\nabla v\|_{H^{1/2}}|w| \;.
\end{equation}
It follows, thanks to \eqref{1_1}, \eqref{grad-nonl},\eqref{Agmon} and \eqref{LadyzhenskayaTilde}, that
\begin{equation*}
\begin{aligned}
\left|A^{1/2}B(u,v)\right| &\leq \sum_{i=1,2}\left|\partial_i (u\cdot \nabla) v\right|=\sum_{i=1,2}\left|(\partial_i u\cdot \nabla) v\right|+
\sum_{i=1,2}\left| (u\cdot \nabla) \partial_i v\right|\\
&\leq c  \sum_{i=1,2}\left(\left\|\partial_i u\right\|_{H^{1/2}}\left\|\nabla v\right\|_{H^{1/2}}+\|u\|_{L^{\infty}} \left|\nabla\partial_i v\right|\right)\\
&\leq c\left(|A^{3/4}u||A^{3/4}v|+|u|^{1/2}|Au|^{1/2}|Av|\right)\;,
\end{aligned}
\end{equation*}
and
\begin{equation*}
\begin{aligned}
\left|A^{1/2}B(u,v)\right|&\leq \sum_{i=1,2}\left|\partial_i (u\cdot \nabla) v\right|=\sum_{i=1,2}\left|(\partial_i u\cdot \nabla) v\right|+
\sum_{i=1,2}\left| (u\cdot \nabla) \partial_i v\right|\\
&\leq c  \sum_{i=1,2}\left(\left|\partial_i u\right|\left\|\nabla v\right\|_{L^{\infty}}+\|u\|_{H^{1/2}} \left\|\nabla\partial_i v\right\|_{H^{1/2}}\right)\\
&\leq c\left(|A^{1/2}u||A^{1/2}v|^{1/2}|A^{3/2}v|^{1/2}+|A^{1/4}u||A^{5/4}v|\right)\;.
\end{aligned}
\end{equation*}
\end{proof}

Inspired by the proof of the Br\'ezis-Gallouet inequality \cite{Brezis-Gallouet} we establish below a  bound for the $L^\infty-$norm, which we will use later to optimize an estimate.

\begin{lemma}
Let $\phi\in \dot{H}_{\rm{per}}^2(\Omega),$ then for every $N\in\bR^+$ we have
\begin{equation}\label{optima}
\begin{aligned}
\left\|\phi\right\|_{L^\infty}\leq \calL_N |\nabla\phi|+
(\sqrt{\pi}\kappa_0N)^{-1}|\Delta\phi| \;,
\end{aligned}
\end{equation}
where
$\calL_N=(8+2\pi\log N)^{1/2}(2\pi)^{-1}$.
\end{lemma}
\begin{proof}
First note that
$$
\sum_{1\le |k|^2 \le N^2} \frac{1}{|k|^2} = 6+ \sum_{3\le |k|^2 \le N^2} \frac{1}{|k|^2}
\le 6 + 4\int_2^N \frac{dx}{x^2} + \int_1^N\int_0^{2\pi} \frac{1}{r^2} r  \ d\theta \ dr \le 8  + 2\pi \log N \;,
$$
and for $N\geq 3$
$$
\sum_{N+1\le |k| } \frac{1}{|k|^4} \le 4\frac{1}{(N+1)^4}+4\int_{N+1}^{\infty} \frac{dx}{x^4}+\int_{N-1}^\infty \int_0^{2\pi} \frac{1}{r^4} r  \ d\theta \ dr \leq \frac{4\pi}{N^2} \;.
$$
Using the Cauchy-Schwarz inequality and Parseval's identity, we have
\begin{equation*}
\begin{aligned}
&\|\phi\|_{L^\infty} \leq\sum_{1 \le  |k|\leq N} |\hat{\phi}_k |+\sum_{|k|\geq N+1}|\hat{\phi}_k |\\
&\le
\left[\sum_{1 \le |k|\leq N}\frac{1}{|k|^2}\right]^{1/2}\left[\sum_{1 \le  |k|\leq N}|k|^2|\hat{\phi}_k|^2\right]^{1/2}+
\left[\sum_{|k|\geq N+1}\frac{1}{|k|^4}\right]^{1/2}\left[\sum_{|k|\geq N+1}|k|^4|\hat{\phi}_k|^2\right]^{1/2}\\
&\le (8 + 2\pi \log N)^{1/2}\frac{|\nabla\phi|}{2\pi}+\frac{2\sqrt{\pi}}{N}\frac{|\Delta\phi|}{2\pi\k0}\;.
\end{aligned}
\end{equation*}

\end{proof}
\section{Determining form and statements of main results\label{Statement-Main-Results}}

\subsection{Interpolant operators} In this subsection we introduce  a unified approach for using  the various determining parameters (modes, nodes, volume elements, etc.) by representing them through interpolant operators that approximate identity.

Let  $J=J_h: (\dot{H}^2_{\rm {per}}(\Omega))^2 \to (\dot{C}^{\infty}_{\rm {per}}(\Omega))^2$ be a finite rank linear operator approximating the identity in the following sense: for every $\phi\in (\dot{H}^2_{per}(\Omega))^2$ we have  $J\phi\in (C^{\infty}_{\rm {per}}(\Omega))^2$ and $J\phi$ has zero spatial average;  in addition, we assume the following hold
\begin{equation}\label{J}
|J\phi-\phi|\leq c_1h|\nabla \phi|+c_2 h^2 |\Delta \phi|,
\end{equation}
\begin{equation}\label{double_J}
\left|\nabla(J\phi-\phi)\right|\leq \tilde{c}_1|\nabla \phi|+\tilde{c}_2 h |\Delta \phi|.
\end{equation}
 Here $h$ is a small parameter that determines the order of approximation.  The rank of $J_h$ is of the order $L/h \ge 1$.  For example, such intepolant polynomials are induced by the determining parameters of the NSE, e.g., determining modes, nodes, volume elements, finite elements projections, etc. (see, e.g., \cite{Cockburn-Jones-Titi,FP,FT2,FT1,Jones-Titi1992,Jones-Titi1993} and references therein). The most straight forward example of such interpolant operators  is the projection operator, $J_h=P_m$, onto the ${\rm{span}}\{w_1,w_2, \cdots w_m\}$, where $h= \lambda_{m+1}^{-1}$. Also, the appendix of \cite{Azouani-Olson-Titi} provides  explicit examples of such interpolant operators that are based on nodal values and are satisfying   \eqref{J} and \eqref{double_J}. We observe that a general framework employing such interpolant polynomials, satisfying \eqref{J}, for investigating the long time dynamics of the NSE was introduced in \cite{Cockburn-Jones-Titi}.

\subsection{Determining form} In this subsection we present a determining form that is induced by the interpolant operators $J_h$. The connection between the long time dynamics of the NSE \eqref{NSE} and the determining form  is achieved through the following:

\begin{proposition}\label{equality}
Let $G \ge 1$, and $u(s)$, for $s\in \bR$, be a solution of the NSE \eqref{NSE} that lies in the global attractor $\cA$. Suppose $w \in C_b(\bR; V)\cap L^2_{\rm{loc}}(\bR; D(A))$ with $\frac{dw}{ds} \in L^2_{\rm{loc}}(\bR;H)$  satisfies the following equation:
\begin{equation}\label{w-Att}
\frac{dw}{ds}+\nu Aw+B(w,w)=f-\mu\nu\kappa_0^2 \cP J(w-u).
\end{equation}
Then we have $w=u$, provided
\begin{align}\label{mucond2}
\mu> 6 c_T G\log (c_3 G) \quad \text{where} \quad c_3=(2c_Tc_0)^{1/3},
\end{align}
and $h$ is small enough to satisfy
\begin{align}\label{mucond1}
2\mu\k0^2c_J h^2 \le 1\;, \quad \text{where} \quad c_J=c_1+\frac{c_2^2}{2}\;,
\end{align}
and $c_T$, $c_0$, $c_1$, and $c_2$ are as in \eqref{Titi_1}, \eqref{new anal}, 
\eqref{J}, and \eqref{double_J}, respectively.
\end{proposition}

The proof of Proposition \ref{equality} is given in section \ref{Prop-equality}. We observe that the existence of solutions to \eqref{w-Att}, as specified in Theorem \ref{equality}, follows from Theorem \ref{wEU}.

Next, we introduce the phase space of the dynamics of our determining form. Let $$
X = C_b^1(\mathbb{R},J(\dot{H}^2_{\rm{per}}(\Omega))^2),
$$
with the norms
 \begin{equation*}
 \begin{aligned}
& \|v\|_X=\sup_{s\in \bR} \|v(s)\|/(\nu\kappa_0)+\sup_{s\in \bR} \|v'(s)\|/(\nu^2\kappa_0^3)\\
& \|v\|_{X,0}=\sup_{s \in \bR} \|v(s)\|/(\nu\kappa_0).
\end{aligned}
 \end{equation*}
Now, let  $v$ be a given element of $X$. We consider the equation
\begin{equation}\label{www}
\frac{dw}{ds}+\nu Aw+B(w,w)=f-\mu\nu\kappa_0^2\cP (Jw-v).
\end{equation}
We will show  that under certain conditions on the parameters $\mu$ and $h$, which depend on $\|v\|_X$,  \eqref{www} has a unique, bounded, global, for all $s\in \mathbb{R}$, solution $w(s)$ as is specified in the following:

\begin{theorem}\label{wEU}
Let $f\in V$, and let $\mathcal{B}_X^\rho(0)=\{v \in X:\, \|v\|_X \le \rho\}$ for some $\rho > 0$.  Fix $K$ and $\mu$ so that
$$
K \ge \big (2\rho^2+ \frac{G^2}{\mu} +1 \big)^{1/2}
$$
and
\begin{align}\label{finalmucond}
c_4 K^2\log \left(c_5 K^2\right)< \mu < 2c_4 K^2\log \left(c_5 K^2\right),
\end{align}
where $c_4=80(c_T+c_B+1)^2$, and $c_5=\sqrt{8}(c_T+c_B+1)$.
Choose $h$ small enough so that
 \begin{align}\label{muhcond}
 2\mu h^2\kappa_0^2\left(c_1^2+c_2\right)<\frac{1}{2}.
 \end{align}
Then for every $v\in \mathcal{B}_X^\rho(0)$ equation \eqref{www} has a unique solution, $w(s)$, that exists globally, for all $s\in \bR$,  and  satisfies the following properties:
\begin{align*}
&\text{(i) } \sup_{s\in \bR} \|w(s)\| \le \nu^2 \k0^2 K^2,\qquad
&\text{(ii) } \sup_{s\in \bR} \|w'(s)\| \le   \nu^2 \k0^3 C(K),\\
&\text{(iii) } \sup_{s\in \bR} |Aw(s)| \le  \nu \k0^2 C(K),\qquad
&\text{(iv) }  \sup_{s\in \bR} |A^{3/2}w(s)| \le \frac{\|f\|}{\nu} +  \nu \k0^3 C(K),\\
& \text{(v) } \sup_{s\in \bR} |Aw'(s)| \le   \k0 C(K)(\|f\|+ \nu^2 \k0^3).
 \end{align*}

 Moreover, suppose that $v_1,v_2 \in \mathcal{B}_X^\rho(0)$ and $w_1, w_2$ are the corresponding solutions of \eqref{www}. Denote by $\gamma=v_1-v_2$ and $\delta= w_1-w_2$.  Then
 \begin{align*}
 &\text{(v) } \sup_{s\in \bR} \|\delta(s)\| \le  4 \nu \k0 \|\gamma\|_X,
 \qquad &\text{(vi) } \sup_{s\in \bR} \|\delta'(s)\| \le   \nu^2 \k0^3 C(K) \|\gamma\|_X, \\
 &\text{(vii) } \sup_{s\in \bR} |A\delta(s)| \le   \nu \k0^2 C(K) \|\gamma\|_X,
 \end{align*}
 where $C(K)= c\exp(cK^2 \log K)$ for some universal constant $c>0$.
\end{theorem}
The proof of Theorem \ref{wEU} will be presented in section \ref{MainProof}.


The following corollary is an immediate consequence of Theorem \ref{wEU}.

\begin{corollary}\label{W-map} Assume  the conditions of Theorem \ref{wEU} hold. Then there exists a Lipschitz continuous map
\[
W:\mathcal{B}_X^\rho(0) \rightarrow C_b^1(\bR;D(A)),
\]
satisfying the following properties: \\
\quad (i) For every $v\in \mathcal{B}_X^\rho(0)$,  $W(v)(s) =w(s)$, for all $s\in \bR$, where $w(s)$ is the unique solution of \eqref{wEU} that corresponds to the input $v(s)$. \\
(ii) For every $v_1,v_2\in \mathcal{B}_X^\rho(0)$
\begin{align*}
\frac{1}{\k0} \sup_{s \in \bR} |A(W(v_1)(s)&-W(v_2)(s))| + \sup_{s \in \bR} \|W(v_1)(s)-W(v_2)(s)\| + \\
 & \frac{1}{\nu \k0^2} \sup_{s \in \bR} \|\frac{d}{ds}(W(v_1)(s)-W(v_2)(s))\| \le \nu \k0 C(K) \|v_1-v_2\|_X.
\end{align*}
\end{corollary}

This map, $W(v)$,  plays a crucial role in the definition of our determining form.
To be more specific, let $u^*$ be a  steady state of the NSE \eqref{NSE}.  Our determining form is the  equation
\begin{equation}\label{_v}
\frac{dv}{dt}=F(v)=-\left\|v-JW(v)\right\|_{X,0}^2(v-Ju^*)\;.
\end{equation}
The precise properties of \eqref{_v} are stated in Theorem \ref{MainThm}, below. But first we need the following

\begin{proposition}\label{J-Att} Suppose  that $G \ge 1$. Then for every $u\in \cA$ we have:
\[
\|J(u)\| \le c^* \nu \k0 G^3 \quad \text{and} \quad \|J(u')\| \le c^* \nu^2 \k0^3 G^7.
\]
Consequently,
\[
\|J(u)\|_X \le c^*G^7=:R.
\]
\end{proposition}
\begin{proof}
 Since $\|J(\phi)\|= |\nabla J(\phi)|$, we apply \eqref{double_J}, and use the fact that by \eqref{muhcond} we have $h\k0 \le 2\pi$, together with Proposition \ref{limsupAu} to conclude the proof.

\end{proof}

\begin{theorem}\label{MainThm}
 Let $G \ge 1$, and suppose that the conditions of Theorem \ref{wEU} hold for  $\rho=4R$, where $R$ is as in Proposition \ref{J-Att}, $R=c^*G^7$.  Then:

  (i)\quad The vector field in the determining form, \eqref{_v}, is a Lipschitz map from the ball $\mathcal{B}^\rho_X(0)=\{v\in X:\|v\|_X < \rho \}$ into $X$.  Hence, equation \eqref{_v} is an ODE, in the space $X$, which has a short time existence and uniqueness for initial data in   $\mathcal{B}^\rho_X(0)$.

  (ii) \quad Moreover,  the ball $\mathcal{B}^{3R}_X(J(u^*))=\{v\in X:\|v- J(u^*)\|_X < 3R \}\subset\mathcal{B}^\rho_X(0)$ is forward invariant in time, under the dynamics of the determining form \eqref{_v}. Consequently, \eqref{_v} has global existence and uniqueness for all initial data in $\mathcal{B}^{3R}_X(J(u^*))$.

(iii) \quad Furthermore, every   solution of \eqref{_v}, with initial data in $\mathcal{B}^{3R}_X(J(u^*))$, converges to the set of steady states of \eqref{_v}.

(iv) \quad All the steady states of  the determining form, \eqref{_v}, that are contained in the ball $\mathcal{B}^\rho_X(0)$ are  given by the form $v(s)=Ju(s)$, for all $s \in \bR$, where $u(s)$ is a trajectory that lies on the global attractor $\mathcal{A}$ of the NSE \eqref{NSE}.

\end{theorem}

\begin{proof}
Showing (i) implies the  short time existence of \eqref{_v}, for initial data in $\mathcal{B}^\rho_X(0)$.  We write $F(v)=-g^2(v)(v-u^*)$
where $g(v)=\|v-JW(v)\|_{X,0}$.  Since
$$
\|F(v_1)-F(v_2)\|_X \le |g^2(v_1)-g^2(v_2)|\|v_1-u^*\|_X +|g^2(v_2)|\|v_1-v_2\|_X\;.
$$
it suffices to show that $g:\mathcal{B}^\rho_X(0) \to \bR$ is  Lipschitz.
Now, we observe that for $v_1,v_2 \in \mathcal{B}^\rho_X(0)$ we have
\begin{align*}
|\|v_1-JW(v_1)\|_{X,0}-\|v_2-JW(v_2)\|_{X,0} | &\le \|v_1-JW(v_1)-[v_2-JW(v_2)]\|_{X,0} \\
&\le \|v_1-v_2\|_{X,0}+\|JW(v_1)-JW(v_2)\|_{X,0}.
\end{align*}
Next we observe, thanks to \eqref{double_J} and the triangle inequality, that
$$
\|JW(v_1)-JW(v_2)\|_{X,0} \le  (\tc_1+1)\sup_{s\in \bR}\|W(v_1)(s)-W(v_2)(s)\| +\tc_2 h\sup_{s\in \bR}|A(W(v_1)(s)-W(v_2)(s))|\;.
$$
By virtue of Corollary \ref{W-map} and the fact that $h\k0 \le 2\pi$ the above implies
$$
|\|v_1-JW(v_1)\|_{X,0}-\|v_2-JW(v_2)\|_{X,0} | \le C(K)\|v_1-v_2\|_X,
$$
where $K$ depends on $\rho=4R$ as specified in Theorem \ref{wEU}. This completes the proof of (i).

Thanks to Proposition \ref{J-Att} we observe that $\mathcal{B}^{3R}_X(J(u*)) \subset  \mathcal{B}^\rho_X(0)$. Thus we have short time existence for \eqref{_v} with initial data in $\mathcal{B}^{3R}_X(J(u^*))$. The proof of (ii) follows from  the dissipativity property of \eqref{_v}, namely: for every $s \in \bR$ fixed, we have
\begin{align*}
\frac{d\|v(t;s)- J(u^*)\|^2}{dt} &= - 2\|v-JW(v)\|_{X,0}^2 \|v(t;s)-J(u^*)\|^2 \\
\frac{d\|v'(t;s)\|^2}{dt} &= -2\|v-JW(v)\|_{X,0}^2 \|v'(t;s)\|^2\;,
\end{align*}
where $'=d/ds$. This property implies that the ball $\mathcal{B}^{3R}_X(J(u*))$ is forward invariant, for all $t \ge 0$, which proves simultaneously  (ii) and (iii) (see justification concerning the steady set of the determining form \eqref{_v}  below).

To prove part (iv) we observe that the steady states  of equation \eqref{_v} in the ball $\mathcal{B}^\rho_X(0)$ are either $v = Ju^*$, or   $v\in \mathcal{B}^\rho_X(0)$ such that $\left\|v-JW(v)\right\|_{X,0}=0$. In the first case, $u^* \in \mathcal{A}$, since $u^*$ is a steady state of the NSE, \eqref{NSE}. In the second case, we have $v(s)=JW(v)(s)$, i.e., $v(s)=J(w(s))$,    for all $s\in \mathbb{R}$, where $w(s)$ is a solution of \eqref{www}. In this case, it follows from equation \eqref{www}  that $w(s)$ is a bounded solution of the NSE, \eqref{NSE}. Therefore,   one concludes, from \eqref{Grashof}, that $w(\cdot)$ is a trajectory on the global attractor $\mathcal{A}$ of the NSE. Conversely,  since  $\rho = 4R$, it follows from Proposition \ref{J-Att} that  $J(\mathcal{A}) \subset \mathcal{B}^{3R}_X(J(u^*)) \subset \mathcal{B}^\rho_X(0)$. Thus, for every trajectory $u(\cdot)\subset \mathcal{A}$ it follows, from Proposition \ref{equality}  and \eqref{www}, that $u(s) = W(Ju)(s)$, for all $s\in \mathbb{R}$. In particulary, $Ju = JW(Ju)$, and consequently $Ju$ is a steady state of \eqref{_v} in $\mathcal{B}^\rho_X(0)$.
\end{proof}

\section{Proof of Proposition \ref{equality}\label{Prop-equality}}
\begin{proof}
Using \eqref{NSE} and \eqref{w-Att} the difference $\de=w-u$ satisfies
\begin{equation*}
\frac{d\de}{ds}+\nu A\de+B(\de,u)+B(u,\de)+B(\de,\de)=-\mu\nu\kappa_0^2 \cP J\de\;.
\end{equation*}
Suppose $\de(\bar s) \ne 0$ for some $\bar s \in  \bR$.   Since $\de( s)$ is a continuous function with values in $V$, then there is some maximal interval
$(s_0,s_1)$, containing $\bar s$, such that $\de(s) \ne 0$ for all $s \in (s_0,s_1)$.
Taking a scalar product with $A\de$ and using \eqref{ortho}, \eqref{moveu}, \eqref{Titi_1} we have
that  for all $s \in (s_0,s_1)$
\begin{align*}
\frac{1}{2}\frac{d}{ds}&\|\de\|^2+\nu |A\de|^2=-(B(\de,u),A\de)-(B(u,\de),A\de)-\mu \nu\k0^2(\cP J\de,\de)\\
=&(B(\de,\de),Au) -\mu\nu\kappa_0^2 ( J\de-\de,A\de)-\mu\nu\kappa_0^2\|\de\|^2\\
 \leq &c_T\|\de\|^2\left(1+\log \frac{|A\de|^2}{\k0^2\|\de\|^2}\right) |Au|
    +\mu\nu\kappa_0^2c_1 h^2|A\de|^2
    + \mu\nu\k0^2c_2 h\|\delta\||A\delta|-\mu\nu\kappa_0^2\|\de\|^2 \\
\leq &c_T\|\de\|^2\left(1+\log \frac{|A\de|^2}{\k0^2\|\de\|^2}\right) |Au|
    +\mu\nu\kappa_0^2c_J  h^2|A\de|^2 -\frac{1}{2}\mu\nu\kappa_0^2\|\de\|^2    \;.
\end{align*}

Thanks to \eqref{mucond1} we have
\begin{align*}
\frac{d}{ds}\|\de\|^2+\nu\k0^2\left[ \mu+\frac{|A\de|^2}{\k0^2\|\de\|^2}-
\frac{2c_T|Au|}{\nu\k0^2}\left(1+\log \frac{|A\de|^2}{\k0^2\|\de\|^2}\right)
\right]\|\de\|^2 \le 0,
\end{align*}
i.e.,
\begin{align}\label{diffineq}
\frac{d}{ds}\|\de\|^2+\nu\k0^2\left[ \mu + \theta-\alpha (1+\log \theta)\right]\|\de\|^2 \le 0\;,
\end{align}
where
$$
\theta= \frac{|A\de|^2}{\k0^2\|\de\|^2} \ge 1 \qquad \text{and} \qquad
\alpha=\frac{2c_T|Au|}{\nu\k0^2}\;.
$$
We now seek a lower bound on $\psi(\theta)= \theta-\alpha (1+\log \theta)$, for $\theta \ge 1$.
Note that
$$
\psi(1)=1-\alpha \qquad \text{and} \qquad \lim_{\theta \to \infty}
\psi(\theta) =\infty;
$$
and that $\psi$ is decreasing for $\theta < \alpha$, and increasing for
$\theta > \alpha$.  Thus
\begin{align*}
 \min_{\theta \ge 1}\psi(\theta)=\begin{cases}
\psi(1)=1-\alpha \ge 0 &\quad  \text{if} \quad  0 \le \alpha \le 1  \\
\psi(\alpha)=-\alpha\log \alpha \ge 0 &\quad  \text{if} \quad \alpha \ge 1.  \\
\end{cases}
\end{align*}
Now note that for $\alpha \in (0,1]$ we have $1-\alpha \ge -\alpha \log
\alpha$.  Indeed, it is easy to check that $\eta(\alpha)=1-\alpha+\alpha\log\alpha$ satisfies
$\eta(0^+)=1$, $\eta(1)=0$, and $\eta'(\alpha)=\log \alpha \le 0$.  We conclude
that
\begin{align}\label{conclude}
 \min_{\theta \ge 1}\psi(\theta)\ge -\alpha \log \alpha\;.
\end{align}
Applying \eqref{conclude} and then \eqref{limsupintAu} to \eqref{diffineq}, we have
since $u \in \cA$ that
\begin{align*}
\frac{d}{ds}\|\de\|^2+\nu\k0^2\left[ \mu -\frac{2c_T|Au|}{\nu\k0^2} \log (2c_Tc_0 G^3)\right]\|\de\|^2 \le 0\;.
\end{align*}
It follows that
\begin{align*}
\|\de(s) \|^2 \le \exp\left\{\left[-\nu\k0^2\mu+ \frac{6c_T\log(c_3G)
}{s-\sigma_0}\int_{\sigma_0}^s |Au| \ d\tau \right] (s-\sigma_0)\right\}\|\de(\sigma_0)\|^2  \;,
\end{align*}
where $c_3=(2c_Tc_0)^{1/3}$ and $s_0<\sigma_0 < s < s_1$ .   If $s_0 > -\infty$, we have $\de(s_0)=0$,
so we take $\sigma_0 \to s_0^+$ and conclude that
$\de(s)=0$ for $s \in (s_0,s_1)$.  Otherwise, we may apply \eqref{limsupintAu}
to obtain
\begin{align*}
\|\de(s) \|^2 \le \exp\left\{\nu\k0^2\left[-\mu+ 6c_T\log(c_3 G)
G \right] (s-\sigma_0)\right\}\|\de(\sigma_0)\|^2  \;,
\end{align*}
for $|\sigma_0|$ large enough.  Taking $\sigma_0 \to -\infty$, we have  by \eqref{mucond2}  that $\de(s)=0$.
Since $s \in (s_0,s_1)$ is arbitrary, in particular $\de(\bar s)=0$, a contradiction.
\end{proof}

\section{Proof of Theorem \ref{wEU}}\label{MainProof}
In this section we will give a formal  proof of each  estimate stated in Theorem \ref{wEU}. However, we will describe below how to give a rigorous justification for the existence of a solution $w$ of \eqref{www} which satisfies, together with $w'$, these estimates. First, one considers   the following Galerkin approximation system for \eqref{www}:
\begin{equation}\label{G1}
\frac{dw_n}{ds}+\nu Aw_n+ P_n B(w_n,w_n)= P_n f-\mu\nu\kappa_0^2 P_n \cP(Jw_n-v),
\end{equation}
 where $P_n$ is the orthogonal projection from $H$ onto $H_n:={\rm{span}}\{w_1,w_2,\cdots,w_n\}$, the first $n$ eigenfunctions of the Stokes operator $A$.

\begin{proposition}\label{Galerkin} Equation \eqref{G1} has a solution $w_n(s)$, for all $s\in \bR$, which satisfies, together with $\frac{dw_n}{ds}(s)$, all the estimates stated in Theorem \ref{wEU}.
\end{proposition}

\begin{proof}
 Let $k\in \mathbb{N}$, we will supplement \eqref{G1} with an initial value $w_n(-k \nu \k0^2) =0$, to obtain a finite system of ODEs with quadratic polynomial nonlinearity. Therefore, \eqref{G1} with this initial data  possesses a unique solution, which we will denote by $w_{n,k}$, in a  small interval of time around the  initial time $s= -k \nu \k0^2$. Moreover,  $\frac{d w_{n,k}}{ds}$ is the unique solution of   the following initial value problem, in this small time  interval about $s= -k \nu \k0^2$,
\begin{equation}\label{G2}
\begin{aligned}
&\frac{d\tilde{w}_n}{ds}+\nu A\tilde{w}_n+ P_n B(\tilde{w}_n,w_n)+ P_n B(w_n,\tilde{w}_n)=-\mu\nu\kappa_0^2 P_n \cP(J\tilde{w}_n-v'),\\
& \tilde{w}_n(-k\nu \k0^2)= P_n f+\mu\nu\kappa_0^2 P_n \cP v.
\end{aligned}
\end{equation}
Focusing on the interval $[-k\nu \k0^2,\infty)$, one can follow the same steps, as below, for establishing the estimates for $\|w\|$  to show that these estimates are also valid for $\|w_{n,k}(s)\|$, for $s\ge -k\nu \k0^2$. Thus $w_{n,k}(s)$ remains bounded, for $s\ge -k\nu \k0^2$, and as a result it solves \eqref{G1}, for  $s\in [-k\nu \k0^2,\infty)$. Moreover, since $|A (\frac{d w_{n,k}}{ds}(-k\nu \k0^2))| = |A(P_n f+\mu\nu\kappa_0^2 P_n \cP v)|$ is finite and  independent of $k$; one can show, following similar steps for bounding $\| w'\|$ and $|A w'|$,  below, that $|A (\frac{d w_{n,k}}{ds}(s))|$ is bounded uniformly, independent of  $k$, for all $s \in [-k\nu \k0^2,\infty)$.

Now, let  $j\in \mathbb{N}$, then by employing the Arzela-Ascoli  compactness theorem   one can extract a subsequence $w_{n,k(j)}$ of $w_{n,k}$ which converges to $w_n^j$, as $k(j) \to \infty$, a solution of \eqref{G1} on the interval $[-j\nu \k0^2,j\nu \k0^2]$. Moreover, $w_n^j(s)$ and  $\frac{dw_n^j}{ds}(s)$ satisfy the estimates stated in Thereom \ref{wEU}, for all $s \in [-j\nu \k0^2,j\nu \k0^2]$. Now by the Cantor diagonal process one can show that $w_{n,k(k)}$ converges to $w_n$, as $k \to \infty$,  and  $w_n$ satisfies the properties stated in the proposition.
\end{proof}

Now, we continue with our justification. Based on Proposition \ref{Galerkin}   we use the Aubin compactness theorem (see, e.g.,  \cite{CF88,T97}) to show that for every $m \in  \mathbb{N}$ there exists a subsequence $w_{n(m)}$, of $w_n$, which converges to $w^{(m)}$    in the relevant spaces on the interval $[-\nu \k0 m, \nu \k0 m]$, as $n(m) \to \infty$.  Moreover, by  passing to the limit, following arguments similar to those for the 2D NSE,  one infers that $w^{(m)}$ is a solution of  \eqref{www} in the interval $[-\nu \k0 m, \nu \k0 m]$; in addition,     $w^{(m)}$ and  $\frac{dw^{(m)} }{ds}(s)$ satisfy  the estimates stated in Theorem \ref{wEU}, on the interval $[-\nu \k0 m, \nu \k0 m]$. Now, we use, once again,  the Cantor diagonal process to show  that the diagonal subsequence,  $w_{n(n)}$, converges, as $n \to \infty$, to a solution $w$ of \eqref{www}; moreover,  $w$ and $w'$, satisfy the estimates  stated in Theorem \ref{wEU}, for all $s \in \bR$. This in turn  concludes the justification of the formal estimates that will be established below.

\subsection{Bound for $\|w\|$}\label{double_w}
Taking the inner product of \eqref{www} with $Aw$, and using \eqref{ortho} we have
\begin{equation*}
\frac{1}{2}\frac{d}{ds}\|w\|^2+\nu |Aw|^2=
(f,Aw)+\mu\nu\kappa_0^2(v,Aw)-\mu\nu\kappa_0^2(Jw-w,Aw)-\mu\nu\kappa_0^2\|w\|^2\;.
\end{equation*}
Thanks to \eqref{J} we obtain
\begin{equation*}
\begin{aligned}
\frac{1}{2}\frac{d}{ds}\|w\|^2+\nu |Aw|^2 \le& \frac{|f|^2}{2\nu}+\frac{\nu}{2}|Aw|^2+\mu\nu\kappa_0^2\|v\|^2+\frac{\mu\nu\kappa_0^2}{4}\|w\|^2-\mu\nu\kappa_0^2\|w\|^2\\
&+\frac{\mu\nu\kappa_0^2}{4}\|w\|^2+ \mu\nu\kappa_0^2 c_1^2 h^2|Aw|^2+\mu\nu\kappa_0^2 c_2 h^2|Aw|^2,
\end{aligned}
\end{equation*}
and thus
\begin{equation}\label{Aw_bound_1}
\begin{aligned}
\frac{d}{ds}\|w\|^2+{\mu\nu\kappa_0^2}\|w\|^2+\nu\left(1-2\mu h^2\kappa_0^2\left(c_1^2+c_2\right)\right) &|Aw|^2\leq \\
&\frac{|f|^2}{\nu}+2\mu\nu\kappa_0^2\|v\|^2\;.
\end{aligned}
\end{equation}
Therefore, if we assume that $h$ is small enough  to satisfy
\begin{equation}\label{h_bound1}
\begin{aligned}
2\mu h^2\kappa_0^2\left(c_1^2+c_2\right)<\frac{1}{2}
\end{aligned}
\end{equation}
we have, thanks to Gronwall's inequality and the assumption that $\|w(s)\|$ is bounded, the following bound
\begin{equation}\label{bound_w0}
\|w(s)\|^2\leq 2\nu^2\kappa_0^2\|v\|_X^2+\frac{|f|^2}{\mu\nu^2\kappa_0^2}\le \nu^2\kappa_0^2 (2\rho^2+\frac{G^2}{\mu})  \le \nu^2\kappa_0^2 K^2.
\end{equation}
Next, we consider the evolution equation  for $w'=dw/ds$:
\begin{equation}\label{www_prime}
\frac{dw'}{ds}+\nu Aw'+B(w',w)+B(w,w')=-\mu\nu\kappa_0^2 \cP(Jw'-v').
\end{equation}
\subsection{Bound for $\|w'\|$}\label{bound_w'_s}
Taking the inner product of \eqref{www_prime} with $Aw'$, and using \eqref{moveu} and \eqref{J} (after following similar steps as above), one obtains
\begin{equation}\label{bound_w'}
\begin{aligned}
\frac{d}{ds}\|w'\|^2+{\mu\nu\kappa_0^2}\|w'\|^2+&2\nu\left(1-\mu h^2\kappa_0^2\left(c_1^2+c_2\right)\right) |Aw'|^2\\
&\leq
2\left|\B{w'}{w'}{Aw}\right|+2\mu\nu\kappa_0^2\|v'\|^2 \\
&\leq 2\left\|w'\right\|_{L^\infty}\|w'\||Aw|+2\mu\nu\kappa_0^2\|v'\|^2\;.
\end{aligned}
\end{equation}
Now, by applying  \eqref{optima} to $\left\|w'\right\|_{L^\infty}$, we obtain
\begin{equation*}
\begin{aligned}
\left\|w'\right\|_{L^\infty}\|w'\||Aw|&\leq
\calL_N\|w'\|^2|Aw|+(\pi^{1/2}\kappa_0N)^{-1}\|w'\||Aw'||Aw|\\
&\le \calL_N \|w'\|^2|Aw|+\frac{1}{\nu\kappa_0^2\pi N^2}\|w'\|^2|Aw|^2+\frac{\nu}{4}|Aw'|^2.
\end{aligned}
\end{equation*}
Hence by \eqref{h_bound1} we have
\begin{equation*}
\begin{aligned}
&\frac{d}{ds}\|w'\|^2+\alpha\|w'\|^2\leq2\mu\nu\kappa_0^2|v'\|^2 \le
\mu\nu^5\kappa_0^8 K^2,
\end{aligned}
\end{equation*}
where
\[
\alpha=\mu\nu\kappa_0^2-2\calL_N|Aw|-\frac{2}{\nu\kappa_0^2\pi N^2}|Aw|^2.
\]

We have from \eqref{Aw_bound_1} and \eqref{h_bound1} that
\begin{equation}\label{prime}
\frac{\nu}{2}|Aw(s)|^2\leq -\frac{d}{ds}\|w(s)\|^2+\mu\nu^3\kappa_0^4 K^2,
\end{equation}
and, thanks to \eqref{bound_w0}, integrating gives
\[
\int_{s-1/\nu\kappa_0^{2}}^s |Aw(\sigma)|^2 d\sigma\leq 2\left(1+\mu\right)\nu\kappa_0^2 K^2\;.
\]
Applying the Cauchy-Schwarz inequality, we have
\[
 \int_{s-1/\nu\kappa_0^{2}}^s |Aw| d\sigma\leq \left(2(1+\mu)\right)^{1/2}K\;,
\]
and hence
\[
\int^s_{s-1/\nu\kappa_0^{2}}\alpha \geq \mu-2^{1/2}\frac{(8+2\pi\log N)^{1/2}}{\pi}(1+\mu)^{1/2}K-\frac{4(1+\mu)}{\pi N^2}K^2.
\]

We want to make sure that
\begin{align}\label{al}
\int^s_{s-1/\nu\kappa_0^{2}}\alpha(\sigma)\,d\sigma\geq 1, \,\, \text{for all} \,\,  s\in \bR\;.
\end{align}
If we choose $N^2=K(1+\mu)^{1/2}$, then \eqref{al} follows from requiring
\begin{align*}
\pi(\mu-1)\ge 2^{1/2}\left[8+ \pi \log K+\frac{\pi}{2}\log(1+\mu)\right]^{1/2}(1+\mu)^{1/2}K+4(1+\mu)^{1/2}K\;.
\end{align*}
It suffices then to have
\[
\pi^2(\mu-1)^2\ge 4\left[8+\pi \log K+\frac{\pi}{2}\log(1+\mu)\right](1+\mu) K^2 +32 (1+\mu) K^2\;,
\]
which is equivalent to
\[
\frac{(\mu-1)^2}{1+\mu} \ge \frac{64}{\pi^2}K^2 + \frac{4}{\pi}K^2\log K +\frac{2}{\pi}K^2\log(1+\mu)\;.
\]
Using the fact that
\[
\frac{(\mu-1)^2}{1+\mu} \ge \frac{1+\mu}{4}   \quad \forall \ \mu \ge 3, \;
\]
it suffices to have
\[
\frac{1+\mu}{4}\left[1-\frac{8}{\pi}K^2\frac{\log(1+\mu)}{1+\mu}\right]\geq \frac{64}{\pi^2}K^2+\frac{4}{\pi}(\log K)K^2 \;.
\]
Since
\[
a\geq 2b\log b,\quad b\geq 9\quad \text{implies}\quad\frac{a}{\log a}\geq b
\]
we set $b=16K^2/\pi$ and $a=1+\mu$, so that if
\begin{align}\label{sweet16}
1+\mu\geq \frac{32K^2}{\pi} \log (16K^2/\pi)\;,
\end{align}
it suffices to have
\[
\frac{1+\mu}{8}\geq \frac{64}{\pi^2}K^2+\frac{4}{\pi}K^2\log K \;.
\]
Thus, to ensure that \eqref{al} holds, we may take
\begin{equation}\label{cond32}
\mu \ge 80 K^2\log K\;.
\end{equation}
 By a similar calculation,
again taking $N^2=K(1+\mu)^{1/2}$,  we have for $0 \le  r \le 1$
 \begin{align*}
 \int_{s-r/\nu\kappa_0^2}^s\alpha(\tau) \, d\tau &\geq r \mu- \left[ \frac{4}{\pi}
- r^{1/2}\frac{2^{1/2}}{\pi}\left(8+\pi\log K +\frac{\pi}{2}\log(1+\mu)\right)^{1/2}\right](1+r\mu)^{1/2}K\;.
 \end{align*}
Thus whenever
\begin{equation}\label{mu_bound3}
\mu\leq c' K^2 \log(c' K)
\end{equation}
 for some $c'$ we have
 \begin{align}
\sup_{0\leq r\leq 1}\exp\left(-\int_{s-r/\nu\kappa_0^2}^s\alpha\right)\le \exp(c K^2\log K) \;,\label{33}
\end{align}
for some absolute constant c.  Ultimately, we will choose $c'$ to be compatible with \eqref{cond32} (see \eqref{finalrepeat}).



\begin{lemma}\label{32}
 Let $\beta \ge 0$ be a constant, and let $y(s)\geq 0$ be an absolutely continuous bounded function satisfying the inequality    $y'+\alpha y\leq\beta$,  for all $s\in \bR$. Suppose also that $\int_{s-1/\nu\kappa_0^2}^s\alpha(\sigma) \, d\sigma \geq 1$, for all $s\in \bR$. Then
 \[
 y(s)\leq 2\frac{\beta}{\nu\kappa_0^2}  \sup_{0\leq r\leq 1}\exp\left(-\int_{s-r/\nu\kappa_0^2}^s\alpha(\sigma) \, d\sigma\right), \,\, \text{for all} \,\, s\in \bR.
\]
 \end{lemma}
 \begin{proof} Note that since $\int_{s-1/\nu\kappa_0^2}^s\alpha(\sigma) \, d\sigma \geq 1$,  we have that $\int_{-\infty}^0\alpha=+\infty$.
Multiplying  $y'+\alpha y\leq\beta$ by the integrating factor $\exp(\int_0^s{\alpha(\sigma) \, d\sigma })$ and integrating from $-\infty$ to $s$ we obtain
 \begin{equation*}
\begin{aligned}
y(s)&\leq \beta\int^{s}_{-\infty} \exp\left(-\int_\sigma^s\alpha \right)d\sigma\\
&=\beta\int_{-\infty}^s\exp\left(-\int_{\sigma+\lfloor(s-\sigma)\nu\kappa_0^2\rfloor/\nu\kappa_0^2}^s\alpha \right)\prod_{k=1}^{\lfloor(s-\sigma)\nu\kappa_0^2\rfloor}\exp\left(-\int_{\sigma+(k-1)/\nu\kappa_0^2}^{\sigma+k/\nu\kappa_0^2}\alpha\right)d\sigma \\
&\le \beta \sup_{0\leq r\leq 1}\exp\left(-\int_{s-r/\nu\kappa_0^2}^s\alpha\right) \int_{-\infty}^s \prod_{k=1}^{\lfloor(s-\sigma)\nu\kappa_0^2\rfloor} e^{-1} d \sigma \\
&= \beta \sup_{0\leq r\leq 1}\exp\left(-\int_{s-r/\nu\kappa_0^2}^s\alpha\right) \int_{-\infty}^s e^{-\lfloor(s-\sigma)\nu\kappa_0^2\rfloor} d \sigma \\
&= \beta \sup_{0\leq r\leq 1}\exp\left(-\int_{s-r/\nu\kappa_0^2}^s\alpha\right)\left(1+e^{-1}+e^{-2}+\dots\right)/\nu\kappa_0^2 \\
&= \frac{2\beta}{\nu\kappa_0^2} \sup_{0\leq r\leq 1}\exp\left(-\int_{s-r/\nu\kappa_0^2}^s\alpha\right)\;.
 \end{aligned}
\end{equation*}
 \end{proof}

 We have from Lemma \ref{32}, \eqref{al}, and  \eqref{33} that
 \begin{equation}\label{wprime}
 \|w'(s)\|\leq c\nu^2\kappa_0^3  \exp\left(cK^2\log K\right), \,\, \text{for all} \,\, s\in \bR, \, \text{for some} \,  c>0.
 \end{equation}



\subsection{Bounds for $|Aw|$ and $|A^{3/2}w|$}\label{bound_A3/2}
From \eqref{prime} we obtain
\[
\frac{\nu}{2}|Aw|^2\leq 2\|w'\|\|w\|+\mu\nu^3\kappa_0^4 K^2\leq \frac{1}{\nu \k0^2}\|w'\|^2+\nu \k0^2 \|w\|^2+\mu\nu^3\kappa_0^4 K^2,
\]
hence, by  \eqref{wprime} and \eqref{bound_w0} we have
\begin{equation}\label{a_exp}
\sup_{s\in \bR} |Aw(s)|\leq c\nu\kappa_0^2\exp(cK^2\log K).
\end{equation}\label{a_exp2}
From \eqref{www} we obtain:
\[
A^{1/2}w'+\nu A^{3/2}w+A^{1/2}B(w,w)=A^{1/2}f-\mu\nu\kappa_0^2 A^{1/2}\cP J w+\mu\nu\kappa_0^2A^{1/2}\cP v,
\]
and thus, using \eqref{1_2} yields
\begin{equation*}
\begin{aligned}
\nu|A^{3/2}w|\leq \|w'\|+\|f\|+\mu\nu\kappa_0^2\|v\|+\mu\nu\kappa_0^2\|\cP Jw\|+
c'\left(|A^{3/4}w|^2+|w|^{1/2}|Aw|^{3/2}\right).
\end{aligned}
\end{equation*}
Since $h\k0 \le c$ we have from \eqref{double_J}
\[
\|\cP Jw\|\leq c\left(\|w\|+h|Aw|\right) \leq c \nu \k0 \exp(cK^2\log K).
\]
Thus, using also \eqref{mu_bound3} and \eqref{a_exp},
\[
\sup_{s\in \bR}|A^{3/2}w(s)|\leq \frac{1}{\nu}\|f\|+\nu\kappa_0^3c\exp(cK^2\log K).
\]
\subsection{Bound for $|Aw'|$}\label{bound_Aw'}

Taking the inner product of $A^2w'$ with the following equation
\begin{equation*}
\frac{dw'}{ds}+\nu Aw'+B(w',w)+B(w,w')=-\mu\nu\kappa_0^2 \cP(Jw'-v'),
\end{equation*}
implies
\begin{equation*}
\begin{aligned}
\frac{1}{2}\frac{d}{ds}|Aw'|^2+\nu|A^{3/2}w'|^2\leq& |(B(w',w),A^2w')|+|(B(w,w'),A^2w')|-\mu\nu\kappa_0^2|Aw'|^2 \\
&+\mu\nu\kappa_0^2\|Jw'-w'\||A^{3/2}w'|+\mu\nu\kappa_0^2\|v'\||A^{3/2}w'|.
\end{aligned}
\end{equation*}
We bound the first nonlinear term using \eqref{Young} and \eqref{1_3} to obtain
\begin{align*}
|(B(w',w),A^2w')|&=|A^{1/2}B(w',w)| |A^{3/2} w'|
  \nonumber \\
 &\le \frac{1}{\nu} |A^{1/2}B(w',w)|^2 + \frac{\nu}{4} |A^{3/2} w'|^2 \\
 &\le \frac{c}{\nu}\left(|A^{1/2}w'||A^{1/2}w|^{1/2}|A^{3/2} w|^{1/2}+|A^{1/4}w'||A^{5/4}w|\right)^2+ \frac{\nu}{4} |A^{3/2} w'|^2
\end{align*}
For the second nonlinear term, we
integrate by parts, using the fact
that $A=-\Delta$ under periodic boundary conditions, so that by \eqref{flip},\eqref{ortho},
\eqref{Ladyzhenskaya}, \eqref{Sobolev}, and \eqref{Young}, we have
\begin{align*}
|(B(w,w'),A^2w')|&= |\int_\Omega  ( (w\cdot \nabla) w') \cdot \Delta^2 w' \ dx|  \nonumber \\
 &\le |\int_\Omega  ( (\Delta w\cdot \nabla) w') \cdot \Delta w' \ dx|+ 2|\sum_{j=1}^2\int_\Omega  ( (\partial_j w\cdot \nabla) \partial_j w' )\cdot \Delta w' \ dx|   \nonumber \\
 &\le c \|\triangle w\|_{L^4} |\nabla w'| \|\Delta w'\|_{L^4}  + c\sum_{j,k=1}^2|\nabla w| \|\partial_j\partial_k w'\|_{L^{4}}^2
  \nonumber \\
 &\le c \|\nabla^2 w\|_{H^{1/2}} \|w'\| \|\Delta w'\|_{H^{1/2}}+ c\sum_{j,k=1}^2|\nabla w| \|\partial_j\partial_k w'\|_{H^{1/2}}^2
  \nonumber \\
 &\le c |Aw |^{1/2} |A^{3/2} w|^{1/2}  \|w'\| |Aw'|^{1/2} |A^{3/2} w'|^{1/2}  + c\| w\| |A w'| |A^{3/2} w'|
  \nonumber \\
 &\le \frac{c}{\nu}|Aw | |A^{3/2} w|  \|w'\|^2 |Aw'| + \frac{c}{\nu}\| w\|^2 |A w'|^2 +
\frac{\nu}{4} |A^{3/2} w'|^2 \\
 &\le \frac{c}{\nu^2\k0^2}|Aw |^2 |A^{3/2} w|^2  \|w'\|^4 + \nu\k0^2(1+cK^2)|Aw'|^2  +
\frac{\nu}{4} |A^{3/2} w'|^2. \label{partsareparts}
\end{align*}

Now, using \eqref{1_3}, \eqref{double_J} and the bounds for $\|w'\|$ and $|A^{3/2}w|$, we obtain by Gronwall's inequality 
the following uniform bound
\begin{equation}\label{a^2_prime_exp}
|Aw'(s)|\leq \kappa_0 c\exp(cK^2\log K)(\|f\|+\nu^2\k0^3), \,\, \text{for all}\,\, s\in \bR,
\end{equation}
provided $\mu \ge 2(1+K^2)$ (which, in turn follows from \eqref{cond32}).
\subsection{Lipschitz property of $w$ and $w'$ in the $D(A)$ norm}\label{ExistSec}
In this section we show that the bounded solutions of \eqref{www} are unique, and depend continuously on the input trajectory $v\in X$, in a sense that will be specified below. In particular, these  properties are instrumental for introducing a well-defined map $v \mapsto W(v)$, from the space $X$ to a space of trajectories, which is defined by $W(v)(s)=w(s)$, for all $s\in \mathbb{R}$.

To achieve these properties one considers the difference, $\delta(s)=w_1(s)-w_2(s)$, between two trajectories $w_1(s)$ and $w_2(s)$, and establishes similar estimates for $\delta$ and $\delta'$ as was done for $w$ and $w'$. Indeed, by linearity the only complication is in the nonlinear term.
Let   $\gamma=v_1-v_2$, and  $\tilde{w}=w_1+w_2$. Then we have
\begin{equation}\label{delta}
\frac{d}{ds}\delta+\nu A\delta +\frac{1}{2}B(\tilde{w},\delta)+\frac{1}{2}B(\delta,\tilde{w})=
-\mu\nu\kappa_0^2\cP\left(J\delta-\gamma\right).
\end{equation}

\subsection{Bound for $\|\delta\|$}
Taking the scalar product of \eqref{delta} with $A\delta$ we obtain, as in Section \ref{double_w},
\begin{equation*}
\begin{aligned}
\frac{1}{2}\frac{d}{ds}\|\delta\|^2&+\nu |A\delta|^2+\mu\nu\kappa_0^2\|\delta\|^2=\mu\nu\kappa_0^2(\gamma,A\delta)-\mu\nu\kappa_0^2(J\delta-\delta,A\delta)\\
&-\frac{1}{2}\left(B(\tilde{w},\delta),A\delta\right)-\frac{1}{2}\left(B(\delta,\tilde{w}),A\delta\right)\\
\le&\  \mu\nu\kappa_0^2\|\gamma\|^2+\frac{\mu\nu\kappa_0^2}{4}\|\delta\|^2+(c_T+c_B)\|\delta\|\|\tilde{w}\|\left(\log\frac{e|A\delta|}{\kappa_0\|\delta\|}\right)^{1/2}|A\delta|\\
&+\frac{\mu\nu\kappa_0^2}{4}\|\delta\|^2+ \mu\nu\kappa_0^2 c_1^2 h^2|A\delta|^2+\mu\nu\kappa_0^2 c_2 h^2|A\delta|^2,
\end{aligned}
\end{equation*}
where for the nonlinear terms we used \eqref{Titi_1} and \eqref{Brezis}.
Applying Young's inequality to the contribution from the nonlinear terms, we have, if \eqref{mucond1} holds, that (using \eqref{bound_w0})
\begin{equation*}
\frac{d}{ds}\|\delta\|^2+\nu\kappa_0^2\left[\mu+\frac{|A\delta|^2}{\kappa_0^2\|\delta\|^2}-8(c_T+c_B)^2 K^2\left(1+\log{\frac{|A\delta|^2}{\kappa_0^2\|\delta\|^2}}\right)\right]\|\de\|^2\leq2\mu\nu\kappa_0^2\|\gamma\|^2
\end{equation*}
At this point we can use \eqref{conclude} with $\theta=|A\delta|^2/\kappa_0^2\|\delta\|^2$ and $\alpha=8(c_T+c_B)^2 K^2$  to obtain
\[
\frac{d}{ds}\|\delta\|^2+\nu\kappa_0^2\left[\mu-8(c_T+c_B)^2 K^2\log \left(8(c_T+c_B)^2 K^2\right)\right]\|\de\|^2\leq2\mu\nu\kappa_0^2\|\gamma\|^2\;.
\]
So if
\begin{equation}\label{mu_bound1}
\mu>32(c_T+c_B)^2 K^2\log \left(\sqrt{8}(c_T+c_B) K\right)\;,
\end{equation}
we have that
\begin{equation}\label{lip_1}
\sup_{s\in \bR}\|\delta(s)\|\leq 4\nu\kappa_0 \|\gamma\|_X.
\end{equation}
To ensure compatibility with \eqref{cond32} and \eqref{mu_bound1} we take,
as in \eqref{finalmucond}
\begin{align}\label{finalrepeat}
c_4 K^2\log \left(c_5 K^2\right)< \mu < 2c_4 K^2\log \left(c_5 K^2\right),
\end{align}
where $c_4=80(c_T+c_B+1)^2$, and $c_5=\sqrt{8}(c_T+c_B+1)$.

\subsection{Bound for $\|\delta'\|$}
The equation for $\delta'$ is
\begin{equation*}
\begin{aligned}
&\frac{d}{ds}\delta'+\nu A\delta' +\frac{1}{2}B(\tilde{w}',\delta)+\frac{1}{2}B(\tilde{w},\delta')+\frac{1}{2}B(\delta',\tilde{w})
+\frac{1}{2}B(\delta,\tilde{w}')=-\mu\nu\kappa_0^2\left(J\delta'-\gamma'\right).
\end{aligned}
\end{equation*}
Taking the scalar product with $A\delta'$, we obtain
\begin{equation*}
\begin{aligned}
\frac{1}{2}\frac{d}{ds}\|\delta'\|^2+\nu |A\delta'|^2&-\frac{1}{2}\left(B(\delta',\delta'),A\tilde{w}\right)+
\frac{1}{2}\left(B(\tilde{w}',\delta),A\delta'\right)+\frac{1}{2}\left(B(\delta,\tilde{w}'),A\delta'\right)\\
&=-\mu\nu\kappa_0^2\|\delta'\|^2+\mu\nu\kappa_0^2(\gamma',A\delta')-
\mu\nu\kappa_0^2(J\delta'-\delta',A\delta').
\end{aligned}
\end{equation*}
Note that the difference with \eqref{bound_w'} (when one changes $w'$ to $\delta'$ and and $w$ to $\tilde{w}$) is in the addition of two terms
$\left(B(\tilde{w}',\delta),A\delta'\right)$ and $\left(B(\delta,\tilde{w}'),A\delta'\right)$, so if we can show that they are bounded by a constant multiple of $\|\gamma\|_X$, we can then obtain the Lipschitz property of $w'$ by applying the same methods as in subsection \ref{bound_w'_s}\;.

We begin with
\begin{align*}
\frac{d}{ds}\|\de'\|^2&+{\mu\nu\kappa_0^2}\|\de'\|^2+2\nu\left[1-\mu h^2\kappa_0^2\left(c_1^2+c_2\right)\right] |A\de'|^2 \nonumber \\
&\le \left|\B{\de'}{\de'}{A\tilde{w}}\right|+
\left|\left(B(\tilde{w}',\delta),A\delta'\right)\right|+\left|\left(B(\delta,\tilde{w}'),A\delta'\right)\right|+2\mu\nu\kappa_0^2\|\gamma'\|^2 .
\end{align*}
We have
\begin{align*}
\left|\B{\de'}{\de'}{A\tilde{w}}\right|&\leq
\left\|\de'\right\|_{L^\infty}\|\de'\||A\tilde{w}| \\
& \le \calL_N\|\delta'\||A\tilde{w}|+(\pi^{1/2}\k0 N)^{-1}\|\delta'\| |A\delta'||A\tilde{w}| \\
& \le \calL_N\|\delta'\||A\tilde{w}|+\frac{\|\delta'\|^2 |A\tilde{w}|^2}{2\pi\nu\k0^2 N^2}+
\frac{\nu}{2}|A\delta'|^2 \;,
\end{align*}
and (from \eqref{A46a})
\begin{equation*}
\begin{aligned}
\left|\left(B(\delta,\tilde{w}'),A\delta'\right)\right| &\leq \csubL |\de|^{1/2}\|\de\|^{1/2}\|\tilde{w}'\|^{1/2}|A\tilde{w}'|^{1/2}|A\de'|\\
&\leq
\frac{\nu}{4}|A\de'|^2+\csubL^2\nu^{-1} \kappa_0^{-2}\|\de\|^2|A\tilde{w}'|^2,
\end{aligned}
\end{equation*}
as well as (from \eqref{A46b})
\begin{equation*}
\begin{aligned}
\left|\left(B(\tilde{w}',\delta),A\delta'\right)\right|\leq \csubA |\tilde{w}'|^{1/2}|A\tilde{w}'|^{1/2}\|\de\||A\de'|\leq
\frac{\nu}{4}|A\de'|^2+\csubA^2 \nu^{-1}\kappa_0^{-2}\|\de\|^2|A\tilde{w}'|^2.
\end{aligned}
\end{equation*}
Thus
\begin{equation*}
\begin{aligned}
\frac{d}{ds}\|\de'\|^2+&{\mu\nu\kappa_0^2}\|\de'\|^2+\nu\left[1-2\mu h^2\kappa_0^2\left(c_1^2+c_2\right)\right] |A\de'|^2\\
&\le \calL_N\|\delta'\||A\tilde{w}|+\frac{\|\delta'\|^2 |A\tilde{w}|^2}{2\pi\nu\k0^2 N^2}+2\mu\nu\kappa_0^2\|\gamma'\|^2+(\csubA^2+\csubL^2)\nu^{-1} \kappa_0^{-2}\|\de\|^2|A\tilde{w}'|^2 \\
&\le \calL_N\|\delta'\||A\tilde{w}|+\frac{\|\delta'\|^2 |A\tilde{w}|^2}{2\pi\nu\k0^2 N^2}+2\mu\nu^5\kappa_0^8\|\gamma\|^2_X+(\csubA^2+\csubL^2)\nu^{-1} \kappa_0^{-2}\|\de\|^2|A\tilde{w}'|^2.
\end{aligned}
\end{equation*}
By \eqref{mucond1}, we may drop the term in $|A\delta'|^2$.  Then using
 \eqref{a^2_prime_exp} on $|A\tilde{w}'|\leq |Aw_1'|+|Aw_2'|$, along with
 \eqref{lip_1}, we have
 $$
  \frac{d}{ds}\|\de'\|^2 + \alpha \|\delta'\|^2 \le \beta\;
  $$
  where
  \begin{align*}
  \alpha=\mu\nu \k0^2 -\calL_N\|\delta'\||A\tilde{w}|+\frac{|A\tilde{w}|^2}{2\pi\nu\k0^2 N^2}
\end{align*}
and
 \begin{align*}
  \beta=2\mu\nu^5\k0^8\|\gamma\|^2_X+32c\nu\k0^2(\csubA^2+\csubL^2)\exp(cK^2\log K)(\|f\|+\nu^2\k0^3)^2\|\gamma\|_X^2\;.
  \end{align*}
 Proceeding as in subsection \ref{bound_w'_s}, we can obtain
\begin{equation*}
\sup_{s\in \bR}\|\de'(s)\|\leq (\|f\|+\nu^2\kappa_0^3) C\|\gamma\|_X,
\end{equation*}
where $C = c\exp(cK^2\log K)$, for some universal constant $c$.
Note that once again \eqref{mu_bound3} suffices to ensure that \eqref{al} holds,
so there is no need to modify the range for $\mu$ in \eqref{mu_bound3}.

\subsection{Bound for $|A\de|$}
We have from \eqref{delta}, using \eqref{A46a}, \eqref{A46b} and \eqref{J}
\begin{equation*}
\begin{aligned}
 \nu|A\de|\leq |\de'|&+\frac{(\csubA+\csubL)}{2\kappa_0}|A\tilde{w}|\|\de\|+
\mu\nu\kappa_0^2|\gamma|\\
&+\mu\nu\kappa_0^2|\de|+\mu h\nu c_1\kappa_0^2\|\de\|+\mu h^2 c_2\nu \kappa_0^2|A\de|
\end{aligned}
\end{equation*}
hence
\begin{equation*}
\begin{aligned}
\nu(1-\mu h^2 c_2 \kappa_0^2)|A\de|\leq \kappa_0^{-1}\|\de'\|+\mu\nu\kappa_0\|\gamma\|
\left[\mu\nu\kappa_0(1+hc_1\kappa_0)+\frac{(\csubA+\csubL)}{2\kappa_0}|A\tilde{w}|\right]\|\de\|
\end{aligned}
\end{equation*}
So if \eqref{mucond1} holds, we have that (by \eqref{a_exp}, \eqref{lip_1})
\begin{equation*}
\sup_{s\in \bR}|A\de(s)|\leq \frac{C}{\nu\kappa_0} (\|f\| + \nu^2\kappa_0^3) \|\gamma\|_X,
\end{equation*}
$C = c\exp(cK^2\log K)$, for some universal constant $c>0$.

\section{Acknowledgements}

 The work of C.F. is supported in part by NSF grant DMS-1109784, that of M.S.J.
by  DMS-1008661 and DMS-1109638 and that of E.S.T.  by DMS-1009950, DMS-1109640 and DMS-1109645, as well as the Minerva Stiftung/Foundation.

\end{document}